\documentclass[11pt]{article}

\usepackage{amssymb, verbatim, amsmath}
\usepackage[all]{xy}
\usepackage{amsthm}
\usepackage[utf8]{inputenc}
\usepackage{dirtytalk}

\parskip \baselineskip

\newcommand{\bc}{\begin{center}}
\newcommand{\ec}{\end{center}}
\newcommand{\be}{\begin{enumerate}}
\newcommand{\ee}{\end{enumerate}}
\newcommand{\beq}{\begin{equation}}
\newcommand{\eeq}{\end{equation}}
\newcommand{\bi}{\begin{itemize}}
\newcommand{\ei}{\end{itemize}}
\newcommand{\bd}{\begin{description}}
\newcommand{\ed}{\end{description}}
\newcommand{\ba}{\begin{array}}
\newcommand{\bea}{\begin{eqnarray*}}
\newcommand{\eea}{\end{eqnarray*}}
\newcommand{\ea}{\end{array}}
\newcommand{\bt}{\begin{tabular}}
\newcommand{\et}{\end{tabular}}

\newcommand{\bmi}{\begin{minipage}}
\newcommand{\emi}{\end{minipage}}

\newtheorem{thm}{Theorem}[section]
\newtheorem{defn}[thm]{Definition}
\newtheorem{lem}[thm]{Lemma}

\newtheorem{cor}[thm]{Corollary}

\begin{document}

\bc {\bf\large On the center of a finite Dickson nearfield}\\[3mm]
{\sc Prudence Djagba  }

\it\small
African Institute for Mathematical Sciences \\
South Africa\\
\rm e-mail: prudence@aims.ac.za
\ec
 
\normalsize

\quotation{\small {\bf Abstract:}   We  study the center  of a finite Dickson nearfield that arises from  Dickson pair $(q,n)$.

\normalsize

\textup{2010} \textit{MSC}: \textup{16Y30;12K05}

\section{Introduction}

Nearfields were first studied by Dickson in $1905$ and were applied immediately by mathematicians to geometry. Lacking one side distributive law makes the study of nearfields difficult despite they look a lot like fields. Zassenhauss \cite{zassenhauss1935},  Dancs \cite{susans1971,susans1972}, Karzel and Ellers \cite{ellerskarzel1964} have solved some important problems in this area. Recently the author in \cite{djagba} has investigated on the generalized distributive set of a finite nearfield\\

In the paper  \cite{ellerskarzel1964}, Eller and Karzel showed that the center of a finite Dickson  that arises from  Dickson pair $(q,n)$ is equal to a finite field of order $q^n$. In the present work, we provide a simple and shortest proof of a result  due to Eller and Karzel on the presentation of the center of a finite Dickson nearfield that arises from a Dickson pair $(q,n)$.\\

Let $S$ be any group with identity  $0$. We will use $S^*$ to denote $S \setminus \{ 0\}.$

\begin{defn}(\cite{meldrum1985near}) Let  $(R,+,\cdot)$ be a triple such that $(R,+)$ is a group,
 $(R,\cdot)$ is a semigroup, and  $a \cdot (b+c)= a \cdot b+a \cdot c$ for all $a,b,c \in R.$ Then $(R,+,\cdot)$ is a (left) nearring. If in addition $ \big ( R^*, \cdot \big )$ is a group then $(R,+, \cdot)$ is called a nearfield. 
\end{defn}

So a nearfield is an algebraic structure similar to a skewfield (sometimes called a division ring) except that it has only one of the two distributive laws.  It is well known that the additive group of a (left) nearfield is abelian, see for instance \cite{zassenhauss1935,  dickson1905finite}. Throughout this paper we will make use of (left) nearfields.

Furthermore, as we know  from the definition of a (left) nearfield, we do not  necessarily have the right distributivity law and  commutativity of  multiplication. For this reason, the following concepts have been defined and they will be used in the next chapters.
\begin{defn}(\cite{pilz2011near})
Let $R$ be a nearfield. 
\begin{itemize}
\item The multiplicative  center of $ (R, \cdot)$ denoted by $C(R)$,  is defined as follows: 
\begin{align*}
C(R)= \left\lbrace x \in R : x \cdot y=y \cdot x  \thickspace  \mbox{for all}\thickspace y \in R \right\rbrace. 
\end{align*}
i.e., it is the set of elements of $R$ that commute with every element of $R$.
\item  We  use $D(R)$ to denote the set of all distributive elements of $R$, also called the kernel of $ (R,+, \cdot)$. It is defined as follows:  
\begin{align*}
D(R) = \lbrace \lambda \in R : \thickspace (\alpha+\beta) \cdot \lambda = \alpha  \cdot \lambda+\beta \cdot \lambda \thickspace \mbox{for all}\thickspace \alpha, \beta  \in R \rbrace.
\end{align*}
\end{itemize}
\label{defdiscen}
\end{defn}

 To construct finite Dickson nearfields, we need two concepts: Dickson pair and coupling map.

\begin{defn} (\cite{pilz2011near})
A pair of positive integers  $(q,n)$ is called a Dickson pair if the following conditions are satisfied:
\begin{enumerate}
\item[(i)] $q$ is some power $p^l$ of some prime $p$,
\item[(ii)] each prime divisor of $n$ divides $q-1$,
\item[(iii)] $q \equiv 3$ $ \text{mod } 4$ implies $4$ does not divide $n$.
\end{enumerate}
\end{defn}

% \begin{defn} (\cite{dickson1905finite})
%A pair of positive integers  $(q,n)$ is called a Dickson pair if the following conditions are satisfied: 
%\begin{enumerate}
%\item[(i)] $q=p^l$ for some  primes $p$,
%\item[(ii)] $\begin{cases}
%\text{for all prime $\pi :$ if $\pi $ divides $n$ then $\pi $ divides $q-1$} \\
%\text{and if $4$ divides $n$ then $4$ divides $q-1$.}
%\end{cases} $ 
%\end{enumerate}
%\end{defn}

Let $(q,n)$ is a Dickson pair  and $k \in \{1, \ldots, n \}.$ We will denote the positive integer  $\frac{q^k-1}{q-1}$ by $[k]_q$. 

 \begin{defn}(\cite{pilz2011near})
Let $R$ be a nearfield and $\textit{Aut} (R,+,\cdot ) $ the set of all automorphisms of $R$. A map $\phi: \thickspace R^* \to  \textit{Aut} (R,+,\cdot ) $  defined by $ a \mapsto \phi_a$
is called a coupling map if for all $a,b \in R^*, \thickspace \phi _a \circ \phi_b= \phi _{ \phi _a (b) \cdot a}.$
\end{defn}

Every Dickson pair $(q,n)$ gives rise to a finite Dickson nearfield. This is obtained by replacing the usual multiplication \say{$ \cdot$} in the finite field $\mathbb{F}_{q^n}$  of order $q^n$ with a new multiplication \say{$\circ$}. We shall denote the set of Dickson nearfields  arising from the Dickson pair $(q,n)$  by $DN(q,n)$ and  the  Dickson nearfield arising from the Dickson pair $(q,n)$ with generator $g$ by $DN_g(q,n)$. Furthermore in \cite{pilz2011near}   the new multiplication is constructed as follows: 

Let $g$ be such that $ \mathbb{F}_{q^n}^*= \langle g \rangle$ and $H = \langle g^n \rangle $. The quotient group is given by
 \begin{align*}
  \mathbb{F}_{q^n}^* / H & = \big \lbrace  g^{[1]_q}H, g^{[2]_q}H,\ldots, g^{[n]_q}H \big \rbrace \\
  & =\big \lbrace  H, gH,\ldots, g^{n-1}H \big \rbrace.
 \end{align*}
 The  coupling map $\phi$ is defined as
 \begin{align*}
\begin{array}{lcl}
\mathbb{F}_{q^n}^*& \to & \textit{Aut}(\mathbb{F}_{q^n},+,\cdot) \\
\alpha  & \mapsto & \phi_{\alpha}= \varphi^k(\alpha)  
 \end{array} 
  \end{align*} where  $\varphi$ is the Frobenius automorphism of $\mathbb{F}_{q^n}$ and $ k$ is a positive integer $( k \in \{1,\ldots,n \})$  such that $ \alpha \in g^{[k]_q}H  $. Let $\alpha, \beta  \in \mathbb{F}_{q^n},$ the  we have 
  \begin{align*}
\alpha \circ  \beta &= \left\{
\begin{array}{lcl}
\alpha \cdot \phi_{\alpha}(\beta) & \text{if} &  \alpha\ \neq 0  \\ 
0 & \text{if} & \alpha=0
\end{array}\right.  \\
&=
\begin{cases}
\alpha \cdot  \varphi^k (\beta) \thickspace  \thickspace \text{if $ \alpha \in g^{[k]_q}H $} \\
 0     \quad \quad \quad \quad  \text{if}  \thickspace     \alpha=0 
\end{cases} \\
  &= 
\begin{cases}
\alpha	 \cdot \beta^{q^k}   \thickspace \thickspace  \text{if $ \alpha \in g^{[k]_q} H  $} \\
0  \quad \quad \quad  \text{if}  \thickspace    \alpha=0  
\end{cases} 
\end{align*}
for $k \in \{ 1,\ldots,n \}$. Thus $DN_g(q,n):=\big ( \mathbb{F}_{q^n}, +, \circ \big )$ is the finite Dickson nearfield constructed by taking $H = \langle g^n \rangle$. By taking all Dickson pairs, all finite Dickson nearfields arise in this way \cite{dickson1905finite}.  Furthermore we deduce the following.

 \begin{lem}(\cite{wahling1987theorie})
 Let $(q,n)$ be a Dickson pair with $q=p^l$ for some prime $p$ and positive integers $l,n$. Let $g$ be a generator of $\mathbb{F}_{q^n}^*$ and  $R$  the finite nearfield constructed with $H = \big < g^n \big >.$ Then $n$ divides $[n]_q$ and $g^{[n]_q}H=H$.
 \end{lem}

\section{Alternative proofs}

 In $1964$ Ellers and Karzel  showed that $C(R) = D(R) \cong \mathbb{F}_q$ where $R$ is a finite Dickson nearfield that arises from the Dickson pair $(q,n)$. Note that $C(R)$ denote the center  and $D(R)$ the set of all distributive elements of $R$. In this section we give an alternative   proof of the fact that $C(R) =\mathbb{F}_q$. \\

The following is well known.
\begin{thm}(\cite{lidl1994introduction})
If $K \subseteq  F$ is a field extension (i.e., $K$ is subfield of $F$), then $F$ is a vector space over $K.$
\end{thm}

\begin{cor}(\cite{lidl1994introduction})
If $K \subseteq F$ is a field extension and $F$ is finite, then $ \vert F \vert = \vert K \vert ^n$ for some $n \in \mathbb{R}.$
\label{c1}
\end{cor}
%\begin{proof}
%We use the fact that a finite-dimensional vector space over $K$ satisfies $F \cong K^n$ for some $n \in \mathbb{N}  $ (the  dimension of $F$ over $K$.)
%\end{proof}
\begin{lem}(\cite{lidl1994introduction})
A polynomial equation of degree $n$ has at most $n$ roots over any field.
\end{lem}
%\begin{proof}
%We use induction and factor out linear terms. Let $p \in F[X], deg(p)=n.$ For $n=1,$ we have $p=aX+b$ where $a \neq 0$ and has exactly one solution $x_0= \frac{-b}{a}.$  Suppose $p \in F[X], deg(p)=n>1$ such that $p(x_0)=0$. Then there exists $Q \in F[X]$ with degree at most $n-1$ such that $p=(X-x_0)Q.$ One has $p(X)=a_nX^n+ \cdots+a_1X +a_0.$ Since $p(x_0)=0$ then 
%\begin{align*}
%p(X) &= p(X)-p(x_0) \\
%&=a_n(X^n-x_0^n)+ \ldots+ a_1(X-x_0) \\
%&= (X-x_0) \big ( a_n(X^{n-1}+ \ldots+x_0^{n-1}) +  \ldots + a_1 \big ).
%\end{align*}
%Then $p$ has $n-1+1=n$ roots.
%\end{proof}

 Furthermore,
\begin{lem} (\cite{lidl1994introduction})
The set of elements fixed by a field automorphism is a field.
\end{lem}
\begin{lem}(\cite{lidl1994introduction})
The subfields of $\mathbb{F}_{p^m}$ are precisely those $\mathbb{F}_{p^t}$ where $t \mid m$ and they  are unique. Furthermore, they are the fields fixed by the automorphism $ \psi ^t : x \mapsto x^{p^t}.$  
\label{lm3}
\end{lem}

% \begin{proof}
%Suppose $K \subseteq \mathbb{F}_{p^m}$ is a subfield. Then the prime field $\mathbb{P}$ of $\mathbb{F}_{p^m}$ is a subfield of $K$ and  $\mathbb{P} \cong \mathbb{F}_{p}, $ we have $ \mid \mathbb{P} \mid =p, $ hence $ \mid K \mid =p^t$ for some $t \in \mathbb{R}$ by Corollary \ref{c1}.
%
%Since $ \mathbb{F}_{p^m}$ is a vector space over $K,$ we have $ \mid \mathbb{F}_{p^m} \mid = p^m = \mid K \mid ^r = p^{tr}$ for some $r \in \mathbb{R},$ i.e., $ t \mid m. $
%
%For uniqueness notice that any $x \in K$ satisfy $x^{\mid K \mid}= x,$ i.e., $x^{p^t}-x=0$ ( in $K$, but also in $ \mathbb{F}_{p^m}$ ). Hence any subfield of size $p^t$ coincides with the set of solutions to $x^{p^t}-x=0$ 
%\end{proof}

 We now deduce the following:
\begin{thm}Let $(q,n)$ be a Dickson pair with $q=p^l$ for some prime $p$ and positive integers $l,n$. Let $g$ be a generator of $\mathbb{F}_{q^n}^*$ and  $R$  the finite nearfield constructed with $H = \big < g^n \big >.$ Let $\mathbb{F}_{q}$ be the unique subfield of order $q$ of $\mathbb{F}_{q^n}$. Then 
\begin{align*}
\mathbb{F}_{q} \subseteq C  (  R  ).
\end{align*}
\label{m1}
\end{thm}
\begin{proof}
By Lemma \ref{lm3}, $\mathbb{F}_{q}$ is the solution set to the equation $x^q-x=0$ in $\mathbb{F}_{q^n}.$ Let $g$ be a generator of $ \mathbb{F}_{q^n}^{*}$ and take $x \in \mathbb{F}_{q}^*$ and write $x=g^l$. Since $x \in \mathbb{F}_{q}, \thickspace$   $ x^q=x,$ i.e., $x^{q-1}=1.$ Then $ \big ( g^l \big ) ^{q-1}= 1,$ i.e., $g^{l(q-1)}=1.$ Thus $|g|=q^n-1$ divides $l(q-1),$ i.e., $[n]_q  \mid l.$ Thus $\mathbb{F}_{q}^{*} = \big <  g^{ [n]_q} \big >$. Since $ n \mid [n]_q$ then $\big <  g^ {[n]_q} \big >$ is a subset of $\big <  g^n \big >$. Thus we have $\mathbb{F}_{q}^{*} \subseteq  H$. Furthermore for $x \in \mathbb{F}_{q}^*, \thickspace $  $x \in H=g^{[n]_q}H$. So by the Dickson construction, $ \phi_x (y)= \varphi^n(y)=y^{q^n}=y,$ hence $\phi_x =id.$  Take any $t \in R$. We have 
\begin{align*}
x \circ t= x \cdot \phi_x(t)=x \cdot t.
\end{align*}
 Moreover, since $x \in \mathbb{F}_q$ then $\varphi(x)=x^q=x.$ Thus $ \varphi ^l(x)=x$ and 
\begin{align*}
t \circ x = t \cdot \phi_t(x)= t \cdot \varphi^l(x)=t \cdot x =x \cdot t.
\end{align*}

Therefore $t \circ x = x \circ t$ for all $t \in R$. So $x \in C(R)$.
\end{proof}
In fact, it is well known in field theory that:
 \begin{thm}(\cite{lidl1994introduction}) Let $F$ be a finite field of order $p^n$ with characteristic $p$ where $p$ is prime. We have $ (a+b)^{p^m}=a^{p^m} + b^{p^m} $ for all $a,b \in F $ and $m \in \mathbb{N}.$
 \label{t:4444}
\end{thm}
We have the following.
\begin{lem}Let $(q,n)$ be a Dickson pair with $q=p^l$ for some prime $p$ and positive integers $l,n$. Let $g$ be a generator of $\mathbb{F}_{q^n}^*$. Then
$\mathbb{F}_{p}  \big < g^n \big > = \mathbb{F}_{q^n}$ where $\mathbb{F}_{p}$ is the unique subfield of $\mathbb{F}_{q^n}$ of order $p$.
\label{l4444} 
\end{lem}
\begin{proof}
 Let $f$ be the smallest positive integer such that $\big ( g^n \big ) ^{p^f}=g^n.$ Then $g^n$ is a solution to the equation $x^{p^f}-x=0$. In fact every $x$ in $\mathbb{F}_{p} \big < g^n \big > $ satisfies $x^{p^f}-x=0.$ We have  $x= \sum _{i \in I} a_ig^{nb_i},$  where $ a_i \in \mathbb{F}_{p}$ and $b_i \in \mathbb{Z}$. Then 
\begin{align*}
x^{p^f} & =  \big (  \sum _{i \in I} a_ig^{nb_i} \big ) ^{p^f} \\
&= \sum_{i \in I} (a_ig^{nb_i})^{p^f}  \quad \mbox{by Theorem} \thickspace  \ref{t:4444} \ \\
&= \sum_{i \in I} a_i^{p^f} \big ((g^{n})^{p^f} \big )^{b_i} \\
&= \sum_{i \in I} a_ig^{nb_i} \\
&=x.
\end{align*} 
Thus  $ \mathbb{F}_{p}\big < g^n \big >  \subseteq  \mathbb{F}_{p^f}$. But note that since $f$ is minimal,   $ \mathbb{F}_{p}\big < g^n \big >  = \mathbb{F}_{p^f}$.

Furthermore, since $(g^n)^{p^f}=g^n,$ we have,
\begin{align*}
(g^n)^{p^f-1}=1 \Leftrightarrow g^{n(p^f-1)}=1,
\end{align*}
hence 
\begin{align*}
|g|= q^n-1 = p^{ln}-1  \mid  n(p^f-1).
\end{align*}
Since $\mathbb{F}_{p}(g^n) = \mathbb{F}_{p^f} \subseteq \mathbb{F}_{p^{ln}}, \thickspace$   $f \mid ln.$ Suppose that $f \neq ln,$ then $f \leq \frac{ln}{2}$ and  
\begin{align*}
p^{ln}-1 \mid n(p^f-1) \Rightarrow p^{ln}-1 \leq n(p^f-1)\leq n( p^{\frac{ln}{2}}-1).
\end{align*}
Dividing by $p^{\frac{ln}{2}}-1,$ we get
\begin{align*}
p^{\frac{ln}{2}}+1 \leq n,
\end{align*}
 but
\begin{align*}
p^{\frac{ln}{2}}+1 \geq 2^{\frac{ln}{2}}+1 \geq 2^{\frac{n}{2}}+1>n.
\end{align*} 
This leads to a contradiction. Thus $f=ln,$ so $ \mathbb{F}_{p}\big < g^n \big >  = \mathbb{F}_{q^n}.$
\end{proof}
\begin{thm}
Let $(q,n)$ be a Dickson pair with $q=p^l$ for some prime $p$ and positive integers $l,n$. Let $g$ be a generator of $\mathbb{F}_{q^n}^*$ and  $R$  the finite nearfield constructed with $H = \big < g^n \big >.$ Let $\mathbb{F}_{q}$ be the unique subfield of order $q$ of $\mathbb{F}_{q^n}$. Then 
\begin{align*}
 C  (  R ) \subseteq \mathbb{F}_{q} .
\end{align*}
\label{m2}
\end{thm}
\begin{proof}
Take $x \in   C \big (  R \big ) $. Then $ x \circ t =t \circ x$ for all $ t \in     R.$ Let $t =g^n \in H.$ Then 
\begin{align*}
t \circ x= t \cdot \phi_t(x)= g^n \cdot \phi_{g^n}(x)=g^n \cdot x \thickspace  \mbox{since}\thickspace  \phi_{g^n}= id.
\end{align*}
Also
\begin{align*}
x \circ t =x \cdot \phi_x(t)=x \cdot \phi_x(g^n).
\end{align*}
Since $x \circ t =t \circ x, \thickspace$  $g^n \cdot x =x \cdot \phi_x(g^n) $. Hence $\phi_x(g^n)=g^n.$ Furthermore, since $\mathbb{F}_p$ is fixed by $\psi$, the Frobenius map,   $\phi_x$ fixes $\mathbb{F}_p$. Therefore $\phi_x$ fixes $\mathbb{F}_p(g^n),$ the smallest subfield of $\mathbb{F}_{q^n}$ that contains $\mathbb{F}_p$ and $g^n.$ By Lemma \ref{l4444}, $ \phi _x$ fixes $\mathbb{F}_{q^n}$. Thus  $ \phi_x=id.$ Now take $t=g \in g^{[1]_q}H,$ So $ \phi_t= \phi_g= \varphi= \psi^l.$ Then 
\begin{align*}
t \circ x =g \circ x = g \cdot \phi_g(x)= g \cdot \varphi (x).  
\end{align*}
Also,
\begin{align*}
x\circ t = x \cdot\phi_x(t)=x \cdot t = x \cdot g.
\end{align*}
Thus $ t\circ x = x \circ t \Leftrightarrow g \cdot \varphi (x)= x \cdot g \Leftrightarrow \varphi(x)=x \Leftrightarrow x^q=x.$ So $x \in \mathbb{F}_q$. \\
\end{proof}

\vspace{10mm}

\section{Concluding comments}
By Theorem \ref{m1} and  Theorem \ref{m2}, we have shown that $C(R) = \mathbb{F}_q$ where $R \in DN(q,n)$. An intersting research line is to characterize all the automorphism of a finite Dickson nearfield.

\end{document}